\documentclass{amsart}
\usepackage{a4wide,amsmath,amssymb}
\usepackage[toc,page]{appendix}
\usepackage{hyperref}

\usepackage{url}

\usepackage{mathtools}

\usepackage{xcolor}

\newtheorem{thm}{Theorem}[section]

\theoremstyle{remark}
\newtheorem{rem}[thm]{Remark}
\newtheorem{opp}[thm]{Open problem}

\theoremstyle{definition}

\newtheoremstyle{Claim}{}{}{\itshape}{}{\itshape\bfseries}{:}{ }{#1}
\theoremstyle{Claim}

\newcommand{\T}{{\mathbb{T}}}

\newcommand{\R}{\mathbb{R}}

\newcommand{\eps}{\varepsilon}

\makeatletter
\@namedef{subjclassname@2020}{\textup{2020} Mathematics Subject Classification}
\makeatother

\theoremstyle{plain}

\begin{document}

\title[]{Quantitative maximal $L^2$-regularity for viscous Hamilton-Jacobi PDEs in 2D and Mean Field Games}

\author{Alessandro Goffi}
\address{Dipartimento di Matematica e Informatica ``Ulisse Dini'', Universit\`a degli Studi di Firenze, 
viale G. Morgagni 67/a, 50134 Firenze (Italy)}
\curraddr{}
\email{alessandro.goffi@unifi.it}

 \thanks{The author is member of the Gruppo Nazionale per l'Analisi Matematica, la Probabilit\`a e le loro Applicazioni (GNAMPA) of the Istituto Nazionale di Alta Matematica (INdAM) and he is partially supported by INdAM-GNAMPA 2026 Project \emph{Processi di diffusione non-lineari: regolarit\`a e classificazione delle soluzioni}
 }
\date{\today}

\subjclass[2020]{35B65,35Q89,35J61}

\keywords{Viscous Hamilton-Jacobi equations, maximal $L^2$-regularity, Mean Field Games with local coupling}

\begin{abstract}
We discuss quantitative Calder\'on-Zygmund estimates in $W^{2,2}$ for 2D viscous Hamil\-ton-Jacobi equations with natural growth in the gradient. We apply the result to obtain the existence of classical solutions for stationary second order Mean Field Games systems in 2D with (defocusing) coupling behaving like $m^\alpha$ for any $\alpha>0$. We also survey on the known results for the regularity of viscous Hamilton-Jacobi equations and second order Mean Field Games and list several open problems.
\end{abstract}

\date{\today}
\maketitle

\section{Introduction}
This note is concerned with nonlinear maximal $L^q$-regularity estimates for equations of the form
\begin{equation}\label{eqHJ}
-\Delta u(x)+|\nabla u(x)|^\gamma=f(x)\text{ in }\Omega\subseteq\R^n,\text{ where }f\in L^q(\Omega) \text{ and }\gamma>1,
\end{equation}
namely
\[
\int_\Omega |D^2u|^q+\int_\Omega |\nabla u|^{\gamma q}\leq C\left(n,q,\gamma,\Omega,\int_\Omega |f|^q\right)\text{ + lower order terms}.
\]
Usually we have poor control of the constant $C$, primarily because of the coercivity of the nonlinearity and the methods of proofs involving either the Bernstein technique (possibly combined with interpolation arguments and duality methods) \cite{CGARMA,CGannpde,GoffiCCM,CGL}, compactness techniques \cite{CirantVerzini,CirantJFA,CirantWei} or the use of the De Giorgi-Nash-Moser/Krylov-Safonov theory \cite{GoffiJFA}. Nonetheless, we observe that for 2D equations with natural growth in the gradient $\gamma=2$ and $q=2$, an approach based on the sole integration by parts provides an accurate evaluation of the constant $C$ of the estimate. This is in line with the treatment of linear boundary value problems, see Chapter 3 of the monograph \cite{Grisvard}. This is done in Theorem \ref{main1}, and the idea of the proof goes back to P.-L. Lions \cite{Napoli}.
We apply this result and some recent estimates obtained in \cite{GoffiJFA} to the regularity theory of stationary Mean Field Games (MFGs) of defocusing type involving coupled second order Hamilton-Jacobi (HJ) and Fokker-Planck (FP) equations on compact manifolds $M\subset\R^2$ without boundary
\begin{equation*}
\begin{cases}
-\Delta u+H(x,\nabla u)+\lambda=m^\alpha&\text{ in }M\subset\R^2,\\
-\Delta m-\mathrm{div}(D_pH(x,\nabla u)m)=0&\text{ in }M\subset\R^2,\\
m>0,\quad \int_{M}m\,dxdy=1,
\end{cases}
\end{equation*}
introduced by J.-M. Lasry and P.-L. Lions \cite{ll}. For these 2D systems with $H$ having natural growth, we show the existence of smooth solutions $(u,\lambda,m)$ \textit{for any $\alpha>0$}; see Theorem \ref{main3}. This smoothness property has been known for a long time by the experts in the community, cf. \cite{LioMinnesota} and the introduction in \cite{Gomessub}, but it has not appeared in print anywhere. We recall that in dimension $n\geq3$ one needs in general certain restrictions on the parameter $\alpha$. The idea of the proof is the following: by second order estimates \cite{Lions2009,Gomesbook,GoffiPediconi} and the 2D Sobolev inequality we get $m^\alpha\in L^q$ for any $q\geq 2$. Then
\begin{align*}
&m^\alpha\in L^2\underset{\text{HJ}}{\overset{\text{maximal reg.}}{\implies}} D^2u,|\nabla u|^2\in L^2\implies -D_pH(x,\nabla u)\in L^r,\ r>2\\
&\underset{\text{FP}}{\overset{\text{linear reg.}}{\implies}} m^\alpha \in C^{\beta}\underset{\text{HJ}}{\overset{\text{Schauder}}{\implies}} u\in C^{2,\beta}\implies -D_pH(x,\nabla u)\in C^{1,\beta}\underset{\text{FP}}{\overset{\text{Schauder}}{\implies}} m\in C^{2,\beta}.
\end{align*}

 Along the way, we provide a short survey of the most recent results about the regularity theory of second order MFGs, pointing out several open problems. A reference for this topic is \cite{Gomesbook}. We do not address here the regularity of first-order MFGs, see \cite{MunozAIHP} for recent developments.\\
 
\textit{Notation}: $u_x,u_y$ will denote the partial derivatives of the unknown function $u=u(x,y)$ in 2D, while $\nabla,D^2$ will denote the gradient and Hessian operators respectively for $n\geq2$. From now on, $\gamma'=\frac{\gamma}{\gamma-1}$ for $\gamma>1$. For the sake of brevity, we will not display the $x$-dependence of $u,\nabla u, D^2u$ and we will sometimes drop the $n$-dimensional Lebesgue measure under the integral sign.

\section{Strong maximal regularity for viscous Hamilton-Jacobi equations in 2D}
Consider
\begin{equation}\label{HJ2d+}
-\Delta u+|\nabla u|^2=f(x)\text{ in }M\subset\R^2,
\end{equation}
or 
\begin{equation}\label{HJ2d-}
-\Delta u-|\nabla u|^2=f(x)\text{ in }M\subset\R^2,
\end{equation}
where $M$ is a compact manifold without boundary. We prove the following Calder\'on-Zygmund estimate via integration by parts, without differentiating the PDE, as it happens for the classical $W^{2,2}$ regularity theory of linear problems \cite[Chapter 3]{Grisvard}. Earlier $W^{2,2}$ results for more general equations of the form $-\mathrm{Tr}(A(x,u,Du)D^2u)+B(x,u,Du)=0$ in two independent variables appeared in \cite[p. 381]{LUbook}: in this case $B$ has quadratic growth and the estimates were obtained under smallness assumptions on $\|u\|_\infty$ via interpolation arguments. 

\begin{thm}\label{main1}
Let $u\in H^1(M)\cap L^\infty(M)$ be a weak solution of \eqref{HJ2d+} or \eqref{HJ2d-} with $f\in L^2(M)$. Then $u\in W^{2,2}(M)$ and the following quantitative regularity estimate holds
\[
\|D^2u\|_{L^2(M)}^2+\||\nabla u|^2\|_{L^2(M)}^2\leq 3\|f\|_{L^2(M)}^2.
\]
\end{thm}
\begin{rem}\label{quantitative}
The $W^{2,q}$ regularity result, $q>\max\left\{\frac{n}{2},2\right\}$, was proposed in \cite{Napoli,LionsSeminar} and it was proved in a non-quantitative form by \cite{CGARMA,GoffiCCM,GoffiJFA,CirantVerzini,CirantWei}. Note that \cite{CGARMA} does not cover the integrability threshold $q=2$, while \cite{GoffiCCM,GoffiJFA} reach the integrability $q=2$. Nonetheless, all the constants of the known results are implicit and depend on $\|f\|_q$, except the (maximal) gradient estimates \cite{CGL} in the case $q\geq n$ and the Lipschitz bounds of the recent paper \cite{AntoniniCianchi}, which hold for quasilinear equations with natural growth in the gradient and right-hand side in more general spaces.  A quantitative estimate as in Theorem \ref{main1} in the general case $f\in L^q$, $n\geq3$ and for general power-growth Hamiltonians remains largely open.
\end{rem}

\begin{proof}\footnote{This result and the idea of the proof was discussed by P.-L. Lions in \cite{Napoli}.}
We consider \eqref{HJ2d+}, the other case being the same. By the a priori estimates in \cite{GoffiCCM} and the existence and uniqueness statements in \cite{BarlesMurat,BarlesPorretta}, weak solutions are strong solutions and they belong to $W^{2,2}$. We show the quantitative $W^{2,2}$ estimate. Using \eqref{HJ2d+}
\begin{equation}\label{f2}
\int_{M}|f|^2\,dxdy\geq \int_{M}(-\Delta u+|\nabla u|^2)^2\,dxdy=\int_{M}[(\Delta u)^2+|\nabla u|^4-2\Delta u|\nabla u|^2]\,dxdy.
\end{equation}
Note now that integrating by parts
\[
\int_{M}u_{xx}u^2_x\,dxdy=\int_{M}u_{yy}u^2_y\,dxdy=0.
\]
Plugging these identities into \eqref{f2} we conclude
\[
\int_{M}|f|^2\geq \int_{M}[(\Delta u)^2+|\nabla u|^4-2u_{xx}u_y^2-2u_{yy}u_x^2]\,dxdy.
\]
Recalling that
\[
|D^2u|^2=u_{xx}^2+u_{yy}^2+2u_{xy}^2
\]
and integrating by parts we get
\[
\int_{M}|D^2u|^2\,dxdy=\int_{M}(u_{xx}^2+u_{yy}^2+2u_{xx}u_{yy})\,dxdy=\int_{M}(u_{xx}+u_{yy})^2\,dxdy=\int_{M}(\Delta u)^2\,dxdy.
\]
Therefore
\[
\int_{M}|f|^2\,dxdy\geq \int_{M}\left(|D^2u|^2+|\nabla u|^4-2u_{xx}u_y^2-2u_{yy}u_x^2\right)\,dxdy.
\]
To conclude, we need to find $C_f>1$ such that
\[
\int_{M}\left(|D^2u|^2+|\nabla u|^4-2u_{xx}u_y^2-2u_{yy}u_x^2\right)\,dxdy\geq \frac{1}{C_f}\int_{M}\left(|D^2u|^2+|\nabla u|^4\right)\,dxdy, 
\]
namely find $C_f>1$ such that
\[
\left(1-\frac{1}{C_f}\right)\int_{M}\left(|D^2u|^2+|\nabla u|^4-2\frac{C_f}{C_f-1}u_{xx}u_y^2-2\frac{C_f}{C_f-1}u_{yy}u_x^2\right)\,dxdy\geq0.
\]
This would imply
\[
\int_{M}\left(|D^2u|^2+|\nabla u|^4\right)\,dxdy\leq C_f\int_{M}|f|^2\,dxdy.
\]
Another integration by parts yields the identity
\[
-2\int_{M}u_{xx}u_y^2\,dxdy-2\int_{M}u_{yy}u_x^2\,dxdy=8\int_{M}u_xu_yu_{xy}\,dxdy.
\]
We write, using that
\[
|\nabla u|^4=(u_x^2+u_y^2)^2=u_x^4+u_y^4+2u_x^2u_y^2,
\]
the following identity
\begin{align*}
&\left(1-\frac{1}{C_f}\right)\int_{M}\left(|D^2u|^2+|\nabla u|^4-2\frac{C_f}{C_f-1}u_{xx}u_y^2-2\frac{C_f}{C_f-1}u_{yy}u_x^2\right)\,dxdy\\
&=\left(1-\frac{1}{C_f}\right)\int_{M}\left[(u_{xx}-u_y^2)^2+(u_{yy}-u_x^2)^2+2u_x^2u_y^2-\frac{2}{C_f-1}u_{xx}u_y^2-\frac{2}{C_f-1}u_{yy}u_x^2+2u_{xy}^2\right]\,dxdy\\
&=\left(1-\frac{1}{C_f}\right)\underbrace{\int_{M}[(u_{xx}-u_y^2)^2+(u_{yy}-u_x^2)^2]\,dxdy}_{\geq0}\\
&+\frac{4}{C_f}\int_{M}\left(2u_xu_yu_{xy}+\frac{C_f-1}{2}u_{xy}^2+ \frac{C_f-1}{2}u_x^2u_y^2\right)\,dxdy.
\end{align*}
Note that by the Young's inequality
\begin{align*}
\frac{C_f-1}{2}u_{xy}^2+\frac{C_f-1}{2}u_x^2u_y^2+2u_xu_yu_{xy}&\geq \frac{C_f-1}{2}u_{xy}^2+\frac{C_f-1}{2}u_x^2u_y^2-2\left(\frac12u_x^2u_y^2+\frac12u_{xy}^2\right)\\
&=\left(\frac{C_f-1}{2}-1\right)u_x^2u_y^2+\left(\frac{C_f-1}{2}-1\right)u_{xy}^2.
\end{align*}
The above term is nonnegative whenever $\frac{C_f-1}{2}-1\geq0$, namely $C_f\geq3$. Therefore
\[
\frac{4}{C_f}\int_{M}\left(2u_xu_yu_{xy}+\frac{C_f-1}{2}u_{xy}^2+ \frac{C_f-1}{2}u_x^2u_y^2\right)\,dxdy\geq0
\]
provided $C_f\geq3$, and the statement follows.
\end{proof}

\begin{rem}[1D HJ equations]
In the case of 1D HJ equations with periodic boundary conditions
\begin{equation}\label{HJode}
-u''(x)+H(u'(x))=f(x),\ x\in[0,L].
\end{equation} 
one can prove that for every convex $\Phi:\R\to\R$ and $\Phi\in C^1$, $H\in C^0$, we have
\[
\int_0^L \Phi(H(u'(x))\,dx\leq \int_0^L \Phi(f(x))\,dx.
\]
This was noticed in \cite{Napoli}. It follows by the convexity of $\Phi$, since
\[
\Phi(-u''+H(u'))\geq \Phi(H(u'))+\Phi'(H(u'))(-u'').
\]
Noting that for $\Psi=\Psi(u')$ we have
\[
\int_0^L \Psi(u')(-u''(x))\,dx=0,
\]
one concludes
\[
\int_0^L\Phi(f)\,dx=\int_0^L\Phi(-u''(x)+H(u'(x)))\,dx\geq \int_0^L\Phi(H(u'(x)))\,dx.
\]
\end{rem}

\begin{rem}
We can handle the case of equations with a zero-th order term $\lambda u$, $\lambda>0$, since
\[
-\Delta u+|\nabla u|^2=f\iff -\Delta u+\lambda u+|\nabla u|^2=\tilde f\text{ where }\tilde f=f+\lambda u.
\]
Whenever $f\in L^q\implies u\in L^q$, the maximal regularity estimate holds equivalently for $\lambda=0$ and $\lambda>0$. The restriction $\lambda>0$ is related to the existence of solutions, but not to a priori estimates \cite{GoffiCCM}.
\end{rem}

\begin{opp}
Many extensions are possible to problems on the whole space $\R^2$ as well as to problems with Dirichlet and Neumann boundary conditions on the line of \cite{Grisvard}.
\end{opp}

\begin{opp}
It is not clear how to extend Theorem \ref{main1} to time-dependent Hamilton-Jacobi equations $\partial_t u-\Delta u+|\nabla u|^2=f\in L^2_{x,t}$. However, if one first proves a quantitative estimate for $\partial_t u\in L^2_{x,t}$, then by freezing the time variable and applying the elliptic result we would get the statement (with a different universal constant). The problem of maximal $L^q$-regularity for time-dependent problems was addressed in \cite{cg20,CGannpde,CirantJFA,GoffiJFA}, but the validity of a strong form as in Theorem \ref{main1} remains open in general. 
\end{opp}

\begin{opp}
The validity of a quantitative estimate for the maximal regularity for \eqref{eqHJ} in the case of general power-like Hamiltonians remains unsettled, cf. Remark \ref{quantitative}.
\end{opp}

\begin{opp}
When $\Delta u$ is replaced with a quasilinear operator modeled on $\Delta_p u$, $p>1$, 
\[
-\Delta_p u+|\nabla u|^\gamma=f,
\]
even the two-dimensional case $n=2$ with natural growth $\gamma=p$ becomes more complicated: some maximal gradient estimates appeared recently in \cite{AntoniniCianchi} for general boundary value problems. In general $n\geq2$ and when $f\in L^q$, $H(x,p)\sim |p|^\gamma$, $\gamma>p-1$, one needs the restrictions $q>\frac{n(\gamma-(p-1))}{\gamma}$ and $p\geq2$ \cite{CGL}, but maximal regularity estimates remain open when $p\leq 2$. The parabolic quasilinear case is largely unknown: maximal first- and second-order regularity hold in a quantitative form by the results in \cite{GoffiLeonoriIUMJ} if $\gamma$ satisfies
\[
0\leq \gamma   <  \max\left\{\frac{p}2,p-1-\frac{p-2}{n+2} \right\},
\]
which can be regarded as a sub-linear growth condition (with respect to the parabolic evolution $\partial_t u-\Delta_p u$). Maximal regularity estimates beyond this threshold are completely open.
\end{opp}

\section{Survey on the classical regularity theory for second order Mean Field Games}
\subsection{Stationary MFGs}
We start by considering a classical stationary second order MFG system with unknown $(u,\lambda,m)$
\begin{equation}\label{MFG1}
\begin{cases}
-\Delta u+H(x,\nabla u)+\lambda=f(m)&\text{ in }\T^n,\\
-\Delta m-\mathrm{div}(D_pH(x,\nabla u)m)=0&\text{ in }\T^n,\\
m>0,\quad \int_{\T^n}m\,dxdy=1.
\end{cases}
\end{equation}
We address three model cases, all of them with $H$ having superlinear growth in the gradient, depending on the behavior of the coupling $f(m)$. We will explore the cases $f(m)=m^\alpha$, $f(m)=-m^\alpha$ and $f(m)=\log m$. We do not consider here other couplings such as those presenting singularities \cite{Gomesbook}, those with aggregation of Choquard-type \cite{BC}, and the case of non-separable Hamiltonians. 

\paragraph{\textit{Stationary defocusing case}} Consider $f(m)=\sigma_f m^\alpha$, $\alpha>0$, $\sigma_f>0$ and $H(x,p)\sim |p|^\gamma$, $\gamma>1$.
\begin{itemize}
\item Existence of smooth solutions for stationary problems of the form \eqref{MFG1} was discussed by P.L. Lions in his seminars around 2010-2012 \cite{Lions2009,LioMinnesota} for $\gamma=2$ and $\alpha<\frac{2}{n-2}$ and, when $n=2$, for any $\alpha>0$. Existence for any $\alpha>0$ was proved for purely quadratic Hamiltonians via the Hopf-Cole change of variable  and the De Giorgi-Nash-Moser theory \cite{NHM} (note that the result therein appeared in the parabolic setting). 
\item Existence of classical solutions also holds for any $\alpha>0$ when $H$ is superlinear and ``slowly increasing'', namely $1<\gamma<\frac{n}{n-1}$, cf. \cite{CirantCPDE}. This is due to the fact that the drift $-D_pH(x,p)\in L^r$, $r>n$, which implies the boundedness of solutions to the second equation of the system via the classical linear theory in \cite{LUbook}.
\item The general case was addressed for problems with periodic boundary conditions in \cite{PimentelIUMJ} when $\alpha<\gamma'/n$, and then in \cite{CirantCPDE}, under the assumption
\[
\alpha<\frac{\gamma'}{n-\gamma'}.
\]
\item The range of $\alpha$ was later boosted via maximal regularity for viscous Hamilton-Jacobi equations (cf. \cite{CGARMA,GoffiPediconi,CirantVerzini,GoffiJFA}) in \cite{GoffiPediconi}: smoothness holds in convex domains with Neumann boundary conditions or compact manifolds without boundary under the condition
\[
\alpha<\frac{\gamma'}{n-2-\gamma'}.
\]
\end{itemize}

\begin{opp}
We expect the existence of classical solutions for any $\alpha>0$ and $\gamma>1$ when $f(m)=\sigma_f m^\alpha$, $\sigma_f$ small. This agrees with the PDE principle that ``smallness'' implies ``smoothness''.
\end{opp}

\begin{opp}
We expect the existence of classical solutions for any $\alpha$ when $f(m)=m^\alpha$ and a large zero-th order coefficient (i.e. $\lambda u$, $\lambda>0$ large). This is motivated by the validity of the smoothness of solutions for short time horizons in the parabolic case \cite{CT,CGSIMA}. Results for related semilinear and fully nonlinear PDEs, along with certain systems,  appeared in \cite{L82book,EvansLions,LionsSouganidis}.
\end{opp}

\begin{opp}
Existence of classical solutions for any $\alpha>0$ and any $\gamma>1$ in compact domains without boundary is open without restrictions on $\gamma$ and $\alpha$; cf. Remark 12 in \cite{NoteCIME}. The best available result for $n\geq3$ in terms of the restrictions of $\alpha$ and $\gamma$ in this setting is \cite{GoffiPediconi} when
\[
\alpha<\frac{\gamma'}{n-2-\gamma'}.
\]
We will show in Theorem \ref{main3} that the smoothness of solutions holds when $\gamma=2$ and $n=2$ without restrictions on $\alpha$.
\end{opp}

\begin{opp}
Existence of classical solutions remains mostly open for MFGs with controlled diffusion (i.e. driven by fully nonlinear or linear nondivergence form operators) and power-like local couplings. The above results cover the cases of constant coefficients and Lipschitz diffusions. We do not know what happens neither for linear uniformly elliptic/parabolic diffusions in divergence and nondivergence form  with bounded and measurable entries nor for fully nonlinear second order terms. Also, the existence of smooth solutions for degenerate diffusions, general power-growth Hamiltonians and local couplings remains largely open.
\end{opp}

\paragraph{\textit{Stationary focusing case}} Consider $f(m)=-\sigma_f m^\alpha$, $\alpha>0$, $\sigma_f>0$ and $H(x,p)\sim |p|^\gamma$, $\gamma>1$.
\begin{itemize}
\item Existence of smooth solutions holds provided that
\[
\alpha<\frac{\gamma'}{n}.
\]
When $\sigma_f$ is small, there exist classical solutions when
\[
\frac{\gamma'}{n}\leq \alpha<\frac{\gamma'}{n-\gamma'}.
\]
See \cite[Theorem 1.1]{CirantCPDE}.
\item In the whole space $\R^n$ there are no smooth solutions when $\alpha>\frac{\gamma'}{n-\gamma'}$, cf. \cite[Theorem 1.3]{CirantCPDE}, but the critical value $\bar \alpha=\frac{\gamma'}{n-\gamma'}$ was not addressed. If the coupling presents a ``small'' space-dependent potential, existence holds up to $\bar \alpha$ by the work \cite{CirantCosenzaVerzini}.
\end{itemize}

\paragraph{\textit{Logarithmic coupling}} Take now $f(m)=\log m$, and $H(p)\sim |p|^\gamma$.
\begin{itemize}
\item Smoothness holds under the assumption 
\[
1<\gamma<2+\frac{1}{n-1},
\]
see \cite{PimentelIUMJ}.
\end{itemize}

\begin{opp}
Smoothness of solutions to \eqref{MFG1} for logarithmic couplings beyond the threshold $\gamma=2+\frac{1}{n-1}$ remains open.
\end{opp}
\subsection{Parabolic MFGs}
Consider now a model parabolic MFG system
\begin{equation}\label{MFG2}
\begin{cases}
-\partial_tu-\Delta u+H(x,\nabla u)=g(m)&\text{ in }\T^n\times(0,T),\\
\partial_t m-\Delta m-\mathrm{div}(D_pH(x,\nabla u)m)=0&\text{ in }\T^n\times(0,T),\\
u(x,T)=u_T(x),\ m(x,0)=m_0(x)&\text{ in }\T^n,
\end{cases}
\end{equation}
where we assume the initial and terminal data $u_T,m_0\in C^{2,\alpha}(\T^n)$.
 
\paragraph{\textit{Parabolic defocusing case}} Consider $g(m)=\sigma_g m^\alpha$, $\alpha>0$, $\sigma_g>0$, and $H(x,p)\sim |p|^\gamma$, $\gamma>1$.
\begin{itemize}
\item Solutions stay always smooth, i.e. in the parabolic H\"older class $C^{2+\beta,1+\beta/2}(\T^n\times(0,T))$, for any $\gamma>1$ and any $\alpha>0$ for short time horizons \cite{CT,CGSIMA} via perturbative methods based on the semigroup theory and interpolation theory in Banach spaces.
\item Solutions are always smooth for weakly coercive Hamiltonians satisfying
\[
1<\gamma<\frac{n+2}{n+1},
\]
see \cite{LioMinnesota,CT} and Remark 12 of \cite{NoteCIME}. In this regime the drift of the Fokker-Planck equations $-D_pH(x,p)\in L^r_{x,t}$, $r>n+2$, which ensures the boundedness and even the H\"older continuity of solutions via the classical theory of linear parabolic equations \cite{LSU}. Then bootstrapping arguments as detailed in the introduction of the paper yield classical regularity. 

\item For $H(p)=|p|^2$, the paper \cite{NHM} provides the smoothness of solutions for any $\alpha$ and $n\geq1$. The proof however applies only to purely quadratic Hamiltonians. 
\item When $H$ has natural gradient growth and $n=1,2$, one has $m^\alpha\in L^q$ for any finite $q>1$ by \cite[Lemma 5.2 with $\mu=2$]{CGannpde} and this implies the smoothness of solutions for any $\alpha>0$.

\item For $H\sim |p|^\gamma$, $\frac{n+2}{n+1}<\gamma<2$ \cite{Gomessub} proved the existence of an implicit exponent $\alpha_{\gamma,n}>\frac{2}{n-2}$ such that there is smoothness for any $0<\alpha<\alpha_{\gamma,n}$.
\item For $H\sim |p|^{2+\mu}$, $\mu\in(0,1)$, \cite{Gomessup} proved the existence of smooth solutions when 
\[
\alpha<\frac{2}{n(1+\mu)-2}.
\]
\item The best known result when $n\geq3$ in terms of the range of $\alpha$ and $\gamma$ for time-dependent defocusing MFGs appeared in \cite{CGannpde} via the parabolic maximal regularity of Hamilton-Jacobi equations and the adjoint method. Therein it is proved the smoothness of solutions when
\begin{equation}\label{bestpar}
\alpha < 
\begin{cases}
\frac{\gamma'}{n-2}\frac{n}{(n+2-\gamma')} & \text{if $\frac{n+2}{n+1} < \gamma \le 2$} \medskip \\ 
\frac{2}{n(\gamma-1) - 2} & \text{if $\gamma \ge 2$}.
\end{cases}
\end{equation}
This result extends \cite{Gomessup} to any $\gamma>2$ with the same upper bound for $\alpha$. In the subquadratic case it provided the first explicit exponent. 
\item When $\gamma>2$ the threshold for $\alpha$ in \eqref{bestpar} can be further improved. Combining \cite[Lemma 5.2]{CGannpde} and the maximal $L^q$-regularity in Theorem 1.1 of \cite{CirantJFA} one can prove the smoothness of solutions when 
\[
\alpha<\frac{\gamma'[(n+2)(\gamma-1)-2]}{(n+2-\gamma')[n(\gamma-1)-2]}\text{ and }\gamma>2.
\]
Note that
\[
\frac{\gamma'[(n+2)(\gamma-1)-2]}{(n+2-\gamma')[n(\gamma-1)-2]}>\frac{2}{n(\gamma-1) - 2}\text{ when }\gamma>2.
\]

\end{itemize}

\begin{opp}
We expect the smoothness of solutions for defocusing MFGs with coupling $f(m)=m^\alpha$, $H$ having natural growth and $n\geq3$ for any $\alpha>0$, but this remains open.
\end{opp}

\begin{opp}
The smoothness of solutions for general Hamiltonians $H(x,p)\sim |p|^\gamma$ and any $\alpha>0$ remains unsettled.
\end{opp}

\begin{opp}
Extensions to parabolic Bellman systems with mean field dependence are possible and largely unexplored; see \cite{BensoussanAMO,BensoussanChinese}.
\end{opp}

\begin{opp}
Smoothness of solutions when the coupling appears also on the terminal data, i.e. $u_T(x)=G(x,m(T))$, is unknown.
\end{opp}

\paragraph{\textit{Parabolic focusing case}} Consider $g(m)=-\sigma_g m^\alpha$, $\alpha>0$, $\sigma_g>0$ and $H(x,p)\sim |p|^\gamma$, $\gamma>1$.
\begin{itemize}
\item The paper \cite{CGannpde} provided at this stage the best available result in terms of $\alpha$ and the growth $\gamma$, proving the smoothness of solutions when \begin{equation*}
\alpha < 
\begin{cases}
\frac{\gamma'}{n} & \text{if $\frac{n+2}{n+1} < \gamma \le 2$} \medskip \\ 
\frac{2}{(n+2)(\gamma-1)-2} & \text{if $\gamma \ge 2$}.
\end{cases}
\end{equation*}
Earlier results appeared in \cite{CT}, where it is proved that classical solutions exist when either $1<\gamma<\frac{n+2}{n+1}$ or under the constraints
\[
\frac{n+2}{n+1}\leq \gamma<2\text{ and }\alpha<\min\left\{\frac{\gamma'}{n},\frac{\gamma'-2}{n+2-\gamma'}\right\}.
\]
\item The more recent paper \cite{CirantGhilli} focused on the case $\alpha\geq\frac2n$, proving the existence of smooth solutions when $\sigma_g$ is small. For large $\sigma_g$ it is known that the existence fails for large time horizons in the whole space $\R^n$.
\end{itemize}

\paragraph{\textit{Parabolic logarithmic case}} When $g(m)=\log m$ the only available result appeared in \cite{GPlog}, where it is proved that smooth solutions exist when
\[
1<\gamma<\frac54.
\]
\begin{opp}
Smoothness of solutions to \eqref{MFG2} for logarithmic couplings beyond the threshold $\gamma=\frac54$ remains open. 
\end{opp}

\section{Applications to stationary second order Mean Field Games in two independent variables}
We consider
\begin{equation}\label{MFG2d}
\begin{cases}
-\Delta u+H(x,\nabla u)+\lambda=f(m)&\text{ in }M\subset\R^2,\\
-\Delta m-\mathrm{div}(D_pH(x,\nabla u)\ m)=0&\text{ in }M\subset\R^2,\\
m>0,\quad \int_Mm\,dxdy=1.
\end{cases}
\end{equation}
We prove the following smoothness result in compact 2D manifolds $M$ without boundary. We assume that $H \in C(M \times \R^2)$ is convex in the second variable, and that there exist constants $C_H>0$ such that
\begin{equation}\label{H}\tag{$H$}
C_H^{-1}|p|^{2}-C_H\leq H(x,p)\leq C_H(|p|^{2}+1)\ ,
\end{equation}
for every $x\in M$, $p\in\R^2$, $X\in\mathrm{Sym}_2$ ($\mathrm{Sym}_2$ being the space of $2\times 2$ symmetric matrices)
\begin{equation}\label{H2}\tag{$H2$}
\begin{gathered}
\mathrm{Tr}(D^2_{pp}H(x,p)X^2) \geq C_H^{-1}|X|^2 -C_H, \\
|D^2_{px}H(x,p)| \leq C_H(|p| + 1), \\
|D^2_{xx}H(x,p)| \leq C_H (|p|^2 + 1).
\end{gathered}
\end{equation}
Concerning the coupling, we will assume here that $f:M\times[0,\infty)\to\R$ is of class $C^1$, and that there exist $\alpha>0$ and $C_f > 0$ such that
\begin{equation}\label{fde}\tag{$M^+$}
C^{-1}_f m^{\alpha-1}  \le f'(m) \le C_f(m^{\alpha-1} + 1) \qquad \forall m \ge 0.
\end{equation}
This implies that $f(\cdot)$ is monotone increasing and bounded from above and below by power-like functions of type $m^\alpha$.
\begin{thm}\label{main3}
Let $n=2$ and assume \eqref{H}-\eqref{H2} and \eqref{fde}. Then there exists a unique smooth solution, i.e. $(u,\lambda,m)\in C^{2,\tilde\beta}(M)\times\R\times C^{2,\tilde\beta}(M)$, of \eqref{MFG2d} for any $\alpha>0$ and some universal $\tilde \beta\in(0,1)$.
\end{thm}
\begin{rem}
The result for \eqref{MFG2d} can be obtained via the Hopf-Cole change of variable and the De Giorgi-Nash-Moser theory when $H(x,\nabla u)=|\nabla u|^2$, following the ideas in \cite{NHM}. The next proof provide a different approach based on nonlinear maximal regularity for HJ equations without decoupling the system, and it provides a stationary counterpart of Theorem 1.4 in \cite{CGannpde} when $n=2$.
\end{rem}
\begin{proof}
By classical second order estimates, i.e. testing the HJ equation against $\Delta m$ and the Fokker-Planck equation against $\Delta u $, cf. \cite{Lions2009,Gomesbook,GoffiPediconi}, we find using \eqref{H2} and \eqref{fde}
\[
\int_M |D^2u|^2m+c_\alpha\int_{M}|\nabla m^{\frac{\alpha+1}{2}}|^2\leq C,
\]
which implies by the 2D Sobolev inequality $m^\alpha\in L^q$ for any finite $q\in(1,\infty)$. In particular, $m^\alpha\in L^2(M)$, which implies by either Theorem \ref{main1} (if $H$ is purely quadratic) or \cite[Remark 3.4, Theorem 4.4 and Remark 4.9]{GoffiJFA} (for a more general $H$) that $u\in W^{2,2}(M)$ and $|\nabla u|^2\in L^2(M)$. This implies that $b(x)=-D_pH(x,\nabla u)\in L^r(M)$, $r>2$. Classical results for linear equations with $L^r$ coefficients \cite[Theorem 5.1 and Theorem 5.2]{LUbook} imply $m\in C^{\beta}(M)$ for some $\beta\in(0,1)$, which in turn gives $m^\alpha\in C^{\tilde \beta}(M)$ for a possibly different $\tilde \beta\in(0,1)$. Then by the Schauder estimates for equations with natural growth \cite[p. 379-380]{LUbook} we find $u\in C^{2,\tilde\beta}(M)$. This implies immediately $b(x)\in C^{1,\tilde\beta}(M)$, which boosts the regularity of the Fokker-Planck equation, giving $m\in C^{2,\tilde\beta}(M)$ after bootstrapping and applying again the Schauder estimates for linear equations in nondivergence form. \\
Once a priori estimates on the second derivatives are established, one can set up a fixed-point method or a regularization procedure which consists in replacing $f(m)$ with $f_\eps(m)=(m\star \chi_\eps)\star\chi_\eps$, where $\chi_\eps$ is a sequence of standard symmetric mollifiers. The existence of solutions with this regularized coupling is standard, cf. \cite{Gomesbook}, and since the previous a priori bounds do not depend on $\eps$, one can pass to the limit and find a solution of \eqref{MFG2d}.
\end{proof}

\begin{opp}
We believe that Theorem \ref{main3}, that holds under natural growth conditions on $H$ and for any $\alpha>0$, can be extended to any $n\geq3$, but this remains open for general natural growth conditions.
\end{opp}


\end{document}